\documentclass[11pt,reqno]{amsart}

\usepackage{amsmath}
\usepackage{amsfonts}
\usepackage{amssymb}
\usepackage{amsthm}
\usepackage{graphicx}
\usepackage{epstopdf}
\usepackage{enumitem}
\usepackage{geometry}
\usepackage{hyperref}
\usepackage{tikz}
\usepackage[all]{xy}

\newtheorem{theorem}{Theorem}
\newtheorem{definition}[theorem]{Definition}

\newtheorem{proposition}[theorem]{Proposition}
\newtheorem{remark}[theorem]{Remark}
\newtheorem{corollary}[theorem]{Corollary}
\newtheorem{lemma}[theorem]{Lemma}

\begin{document}

\title{A notion of continuity in discrete spaces and applications}

\author{Valerio Capraro}
\address{University of Neuchatel, Switzerland}
\thanks{Supported by Swiss SNF Sinergia project CRSI22-130435}
\email{valerio.capraro@unine.ch}



\date{}

\maketitle

\begin{abstract}
We propose a notion of continuous path for locally finite metric spaces, taking inspiration from the recent development of A-theory for locally finite connected graphs. We use this notion of continuity to derive an analogue in $\mathbb Z^2$ of the Jordan curve theorem and to extend to a quite large class of locally finite metric spaces (containing all finite metric spaces) an inequality for the $\ell^p$-distortion of a metric space that has been recently proved by Pierre-Nicolas Jolissaint and Alain Valette for finite connected graphs.
\end{abstract}

\section{Introduction, main results and motivations}

The A-theory is a homotopy theory for locally finite connected graphs that has been developed by Barcelo et alii in a recent series of three papers \cite{Ba-Kr-La-We01},\cite{Ba-La05},and \cite{Ba-Ba-Lo-Ra06}. This theory is based on ideas that go back to the work of Atkin\cite{At74},\cite{At76} - indeed, the letter A is in honour of Atkin - and that were already re-explored in \cite{Kr-La98}. This theory is based on a notion of continuous path that makes sense for all locally finite metric spaces\footnote{A metric space is called locally finite if every bounded set is finite.}.

\begin{definition}\label{def:discretepaths}
Let $(X,d)$ be a locally finite metric space. Given $x\in X$, denote by $dN_1(x)$ the smallest closed ball with the center $x$ which contains at least two points. A finite sequence of points in $X$, say $x_0,x_1,\ldots, x_{n-1},x_n$, is called a continuous path if
$$
x_i\in dN_1(x_{i-1})\qquad\text{and }\qquad x_{i-1}\in dN_1(x_i)\qquad\text{for all }i=1\ldots n
$$
A locally finite metric space is called path-connected if any pair of points can be joined by a continuous path.
\end{definition}

As we will show in the next sections, the main interest of this definition is that it allows, on one hand, to bring down results from Topology of Manifolds to locally finite spaces (as the Jordan curve theorem); on the other hand, it allows to bring up results from Graph Theory to locally finite metric spaces (as the P.N.Jolissaint-Valette inequality).

We now state our main results. For the exact definitions, we refer the reader to the next sections.

\begin{theorem}\label{th:jordan1}
Let $\gamma$ be a simple circuit in $\mathbb Z^2$. Then $\mathbb Z^2\setminus\gamma$ has two path connected components, one finite and one infinite, and $\gamma$ is the boundary of each of them.
\end{theorem}

The reason behind the choice of this application is that the Jordan curve theorem in $\mathbb Z^2$ is of interest in Digital Geometry, a branch of Theoretical Computer Science that studies, roughly speaking, the geometry of the screen of a computer. The basic idea is that the screen of a computer is modeled by the grid $\mathbb Z^2$, whose points are pixels, and one has to turn on some pixels in order to form an image. In this context, the importance of a Jordan curve theorem is clear. For an introduction to Digital Geometry, see the beautiful introduction and Sec. 1.2 of Melin's Ph.D. thesis\cite{Me08}; references on digital version of the Jordan curve theorem include, for instance, \cite{Kh-Ko-Me90},\cite{Sl04},\cite{Bo08} and references therein. Our proposal of a Jordan curve theorem is different from those ones, since it is based on a new definition of a simple curve and on a different notion of connectivity. The same notion of connectivity has been considered in \cite{Ki00} and \cite{Sl06}, where also versions of the Jordan curve theorem have been presented. In particular, in \cite{Ki00} a Jordan curve theorem has been proved for the so-called strict curves. We will see that every strict curve is also simple and that there are simple curves that are not strict (see Remark \ref{rem:comparison}). 

Our second main result is the following

\begin{theorem}\label{th:jolivale1}
Let $X$ be a locally finite metric space such that every path-connected component $X_i$ is finite. The $\ell^p$-distortion of $X$ has the following lower bound
\begin{align}
c_p(X)\geq\sup_i\frac{D(X_i)}{2d(X_i)}\left(\frac{|X_i|}{|E(X_i)|\lambda_1^{(p)}}\right)^{\frac{1}{p}}
\end{align}
\end{theorem}

This inequality was recently proved for finite connected graphs by Pierre-Nicolas Jolissaint and Alain Valette (see \cite{Jo-Va11}, Theorem 1). Here we propose a generalization that holds for the class of locally finite spaces such that each path-connected component is finite. Of course, this class contains all finite connected graphs.\\

\textbf{Acknowledgements:} We would like to thank Pierre-Nicolas Jolissaint for useful comments on Sec. \ref{se:jolivale} and a referee for helpful comments to improve the exposition of the paper.

\section{The Jordan curve theorem in $\mathbb Z^2$}\label{se:jordan}

The classical Jordan curve theorem states that a simple closed curve $\gamma$ in $\mathbb R^2$ separates $\mathbb R^2$ in two path-connected components, one bounded and one unbounded and $\gamma$ is the boundary of each of these components. In this section we want to prove an analogous result in $\mathbb Z^2$ with the Euclidean distance. One is tempted to define a simple circuit in $\mathbb Z^2$ as a continuous circuit $x_0x_1\ldots x_{n-1}x_n$ such that the $x_i$'s are pairwise distinct for $i\in\{0,1,\ldots,n-1\}$ and $x_0=x_n$. With this definition the Jordan curve theorem would be false: consider the following continuous circuit:
$$
(0,0)(0,-1)(1,-1)(2,-1)(2,0)(2,1)(1,1)(1,2)(0,2)(-1,2)(-1,1)(-1,0)(0,0).
$$
This circuit is \emph{simple} in the previous sense, but it separates the grid $\mathbb Z^2$ in three path-connected components: in some sense, this circuit behaves like the $8$-shape curve in $\mathbb R^2$, which is not simple. A discrete analogue of a simple circuit is, in our opinion, something different.

First of all, we need to introduce some notation. Let $(x,y)\in\mathbb Z^2$,
\begin{itemize}
\item $dB_1(x,y)$ denotes the set $\{(x+1,y),(x-1,y),(x,y-1),(x,y+1)\}$,
\item $dB_2(x,y)$ denotes the set $\{(x-1,y-1),(x-1,y+1),(x+1,y+1),(x+1,y-1)\}$.
\end{itemize}

\begin{definition}\label{def:simplecurve}
A simple circuit in $\mathbb Z^2$ is a continuous path $x_0x_1\ldots x_{n-1}x_n$ such that
\begin{itemize}
\item The $x_i$'s are pairwise distinct for $i\in\{0,1,\ldots,n-1\}$ and $x_0=x_n$.
\item Whenever $x_i\in dB_2(x_j)$, for $j>i$, then $j=i+2$ and
$$
x_{i+1}\in dB_1(x_i)\cap dB_1(x_j)
$$
\end{itemize}
\end{definition}

\begin{remark}\label{rem:comparison}
{\rm A Jordan curve theorem in $\mathbb Z^2$ with the same notion of connectivity as ours has been proved in \cite{Ki00} for the so-called strict circuits. They are circuits $\gamma=x_0x_1\ldots x_{n-1}x_n$ such that $|dB_2(x_i)|=2$, for all $i$. Consequently, every strict circuit is also simple. The converse is not true, since the circuit
$$
\gamma=(1,0)(1,1)(0,1)(-1,1)(-1,0)(-1,-1)(0,-1)(1,-1)(1,0)
$$
is simple but not strict, since $|dB_2(1,0)|=3$.}
\end{remark}

For simplicity, we divide the proof of the Jordan curve theorem in $\mathbb Z^2$ in two parts: in Theorem \ref{th:jordan}, we prove that $\mathbb Z^2\setminus\gamma$ has two path-connected components, one finite and one infinite; in Proposition \ref{prop:boundary}, after defining a good notion of boundary, we prove that $\gamma$ is the boundary of each of these path-connected components.

\begin{theorem}\label{th:jordan}
Let $\gamma$ be a simple circuit in $\mathbb Z^2$ not containing squares\footnote{A square in $\mathbb Z^2$ is just a set of four points of the shape $(x_0,y_0),(x_0+1,y_0),(x_0+1,y_0+1),(x_0,y_0+1)$. This hypothesis has an explanation in terms of A-theory. In this theory, squares are homotopic equivalent to one point and therefore they do not contribute in disconnecting the grid $\mathbb Z^2$.}. Then $\mathbb Z^2\setminus\gamma$ has two path connected components, one finite and one infinite.
\end{theorem}

\begin{proof}
Embed canonically $\mathbb Z^2$ into $\mathbb R^2$ and construct the following subset $\tilde\gamma$ in $\mathbb R^2$:
\begin{itemize}
\item $\tilde\gamma$ contains all points contained by $\gamma$.
\item If two points of $\gamma$ are adjacent vertices of a square, then $\tilde\gamma$ contains the segment connecting these two points.
\end{itemize}
It is easy to see that $\tilde\gamma$ can be realized as a closed curve in $\mathbb R^2$ which is simple in the standard sense. Indeed it is made by sides of squares, without repetitions. By the classical Jordan curve theorem, let $Int(\tilde\gamma)$ and $Ext(\tilde\gamma)$ be the two path connected components of $\mathbb R^2\setminus\tilde\gamma$. Define $Int(\gamma)=Int(\tilde\gamma)\cap\mathbb Z^2$ and $Ext(\gamma)=Ext(\tilde\gamma)\cap\mathbb Z^2$ and let us prove that these are exactly the two path connected components of $\mathbb Z^2\setminus\gamma$. Of course, it is enough to show that both $Int(\gamma)$ and $Ext(\gamma)$ are path-connected, since the sets $Int(\gamma),Ext(\gamma),\gamma$ form a partition of $\mathbb Z^2$. So, let us start proving that $Int(\gamma)$ is path connected. Since $Int(\tilde\gamma)$ is bounded, we need first to prove that $Int(\gamma)$ is not empty. Let $(x_0,y_0)\in Int(\tilde\gamma)$. By continuity, we may suppose that $(x_0,y_0)$ is in the interior of some square with vertices in $\mathbb Z^2$. Let $(x_1,y_1),(x_1+1,y_1),(x_1+1,y_1+1),(x_1,y_1+1)$ be such vertices. Observe that at least one of these points must belong outside of $\gamma$, since $\gamma$ does not contain squares. This point has to belong also in $Int(\tilde\gamma)$, since the segment line connecting this point to $(x_0,y_0)$ does not intersect $\tilde\gamma$. Therefore, we have found a point in $Int(\gamma)$. Now we prove that $Int(\gamma)$ is path-connected. Let $(w_0,z_0),(w_1,z_1)\in Int(\gamma)$ and let $\tilde\delta$ be a continuous path in $Int(\tilde\gamma)$ connecting them.  Let $I_i,i=0,\ldots N$, be a covering of the interval $[0,1]$, made of intervals $I_i=[a_i,b_i]$, such that
\begin{itemize}
\item $a_0=0$ and $b_N=1$,
\item $a_i<a_{i+1}$, for all $i$,
\item $a_i=b_{i-1}$, for all $i=1,\ldots N$,
\item For all $i=0\ldots,N$, $\tilde\gamma|_{I_i}$ belongs to some closed squares with vertices in $\mathbb Z^2$,
\item For all $i=0,\ldots, N$ both $\tilde\gamma(a_i)$ and $\tilde\gamma(b_i)$ belongs to the side of some square with vertices in $\mathbb Z^2$.
\end{itemize}
It is easy to construct explicitly such a covering, making use of continuity of $\tilde\gamma$ and compactness of the interval $[0,1]$.
Now, lift $\tilde\delta$ to a discrete path $\{(c_k,d_k)\}$ as follows:
\begin{itemize}
\item Of course, define $(c_0,d_0)=\tilde\delta(0)=(w_0,z_0)$,
\item If $\tilde\delta(\frac1N)\in\mathbb Z^2$, we have two sub-cases:\\
\begin{itemize}
\item If $\tilde\delta(\frac1N)$ is adjacent\footnote{Recall that we are working inside a square and therefore adjacent vertices are exactly the extremal point of a side of the square and opposite vertices are the extremal points of a diagonal of the square.} to $(w_0,z_0)$, define $(c_1,d_1)=\tilde\delta(\frac1N)$,
\item If If $\tilde\delta(\frac1N)$ is opposite to $(w_0,z_0)$, first define $(c_2,d_2)=\tilde\delta(\frac1N)$ and then observe that the two points of $dB_1(c_0,d_0)\cap dB_1(c_2,d_2)$ cannot belong both to $\gamma$, since $\gamma$ is simple. Let, $(c_1,d_1)$ be the one that does not belong to $\gamma$. To see that it is the right choice, it suffices to observe that $(c_1,d_1)\notin Ext(\tilde\gamma)$. To see this, just observe that the segment line connecting $(c_1,d_1)$ to $(c_2,d_2)$ does not intersect $\tilde\gamma$ and therefore $(c_1,d_1)\notin Ext(\tilde\gamma)$.
    \\
\end{itemize}

Now suppose that $\tilde\delta(\frac1N)\notin\mathbb Z^2$. Since $\tilde\delta(\frac1N)\in Int(\tilde\gamma)$, in particular $\tilde\delta(\frac1N)\notin\tilde\gamma$. It follows that there is at least one extremal vertex of the (unique) side of $\mathbb Z^2$ containing $\tilde\delta(\frac1N)$ which does not belong to $\gamma$ (otherwise, by construction of $\tilde\gamma$, we would have $\tilde\delta(\frac1N)\in\tilde\gamma)$. This vertex, now denoted by $(c_2,d_2)$ has to belong to $Int(\gamma)$, since the segment line connecting $(c_2,d_2)$ to $\tilde\delta(\frac1N)$ does not intersect $\tilde\gamma$. Now, if $(c_2,d_2)$ is adjacent to $(c_0,d_0)$, we can just define $(c_1,d_1):=(c_2,d_2)$; otherwise, observe that one of the two points in $dB_1(c_0,d_0)\cap dB_1(c_2,d_2)$ has to belong to $Int(\gamma)$ and we can pick $(c_1,d_1)$ to be one of them.

    \item And so on, for all $i$ up to $N$.
\end{itemize}
In a similar way one shows that $Ext(\gamma)$ is path-connected.
\end{proof}

Now, in order to have a complete analogue of the classical Jordan curve theorem we need to prove a discrete analogue of the fact that $\gamma$ is the boundary of each of the path-connected components of $\mathbb R^2\setminus\gamma$. In order to do that, we use a property that in the classical setting is implied by injectivity.\\

In order to describe this property, let $\gamma$ be a simple circuit in $\mathbb R^2$ in classical sense. Let $Int(\gamma)$ be the bounded path-connected component of $\mathbb R^2\setminus\gamma$. For any point $(x_0,y_0)\in Int(\gamma)$, define four points as follows
\begin{itemize}
\item $(x_0^+,y_0)$ is the first point where the horizontal half-line $x\geq x_0$, $y=y_0$, intersects $\gamma$,
\item $(x_0^-,y_0)$ is the first point where the horizontal half-line $x\leq x_0$, $y=y_0$, intersects $\gamma$,
\item $(x_0,y_0^+)$ is the first point where the vertical half-line $x=x_0$, $y\geq y_0$, intersects $\gamma$,
\item $(x_0,y_0^-)$ is the first point where the vertical half-line $x=x_0$, $y\leq y_0$, intersects $\gamma$.
\end{itemize}

Making this procedure for each point belonging to the path-connected component containing $(x_0,y_0)$, we get the whole $\gamma$. Observe that this would have been false if $\gamma$ were not simple. Our definition of simplicity is exactly the one that makes this procedure working in our discrete world of $\mathbb Z^2$. Indeed, if now $\gamma$ is a simple circuit in $\mathbb Z^2$, the previous construction gives a set which is in general smaller than $\gamma$, because of \emph{angles}, but, if $\gamma$ is simple, it can be \emph{completed} without ambiguity. Formally, given a point $(x_0,y_0)\in Int(\gamma)$, construct four points as before and denote by $A$ the set of points obtained by making this procedure for all points belonging to the path-connected component containing $(x_0,y_0)$. Now, if $(x_1,y_1),(x_2,y_2)\in A$ are opposite vertices of some square, our hypothesis that $\gamma$ is simple and does not contain squares, tells us that there is a unique common adjacent point to both $(x_1,y_1),(x_2,y_2)$ belonging to $\gamma$. Denote by $\overline{Int(\gamma)}$ the set obtained adding to $A$ all these \emph{angle points}. By an analogous construction starting from $(x_0,y_0)\in Ext(\gamma)$, we may define $\overline{Ext(\gamma)}$.\\

The following proposition is now straightforward and concludes our proposal of a discrete analogue of the Jordan curve theorem in $\mathbb Z^2$.

\begin{proposition}\label{prop:boundary}
Let $\gamma$ be a simple circuit in $\mathbb Z^2$ that does not contain squares. Then
$$
\overline{Int(\gamma)}=\overline{Ext(\gamma)}=\gamma
$$
\end{proposition}

\begin{proof}
Just observe that $\overline{Int(\gamma)}$ and $\overline{Ext(\gamma)}$ are respectively the boundary of $Int(\tilde\gamma)$ and $Ext(\tilde\gamma)$, where $\tilde\gamma$ is constructed as in the proof of Theorem \ref{th:jordan}.
\end{proof}

\section{P.N.Jolissaint-Valette's inequality for finite metric spaces}\label{se:jolivale}

The theory of (approximate) embedding of metric spaces in some other well-understood metric space, as a Banach space or a Hilbert space, is now a widely explored field of research, after the breakthrough papers of Linial-London-Rabinovich\cite{Li-Lo-Ra95} and Yu\cite{Yu00}, that found relations among it, Theoretical Computer Science and $K$-theory of $C^*$-algebras. One of the most basic notions in this theory is the notion of $\ell^p$-distortion, which measures how badly a metric space can be embedded in an $\ell^p$-space in a bi-lipschitz way.

For the convenience of the reader we recall that a bi-lipschitz embedding of a metric space $(X,d)$, in this context, is a mapping $F:X\rightarrow\ell^p$ such that there are constants $C_1,C_2$ such that for all $x,y\in X$ one has
$$
C_1d(x,y)\leq d_p(F(x),F(y))\leq C_2d(x,y)
$$
where $d_p$ stands for the $\ell^p$-distance. It is clear that $F$ is injective and so we can consider $F^{-1}:F(X)\rightarrow X$. Therefore, the following notation makes sense,
$$
||F||_{Lip}=\sup_{x\neq y}\frac{d_p(F(x),F(y))}{d(x,y)}
$$
and
$$
||F^{-1}||_{Lip}=\sup_{x\neq y}\frac{d(x,y)}{d_p(F(x),F(y))}
$$
The product $||F||_{Lip}||F^{-1}||_{Lip}$ is called \emph{distortion} of $F$ and denoted by $Dist(F)$.

\begin{definition}\label{def:distortion}
The $\ell^p$-distortion of a metric space $X$ is the following number

\begin{align}\label{eq:distortion}
c_p(X)=\inf\left\{Dist(F) : F \text{ is a bi-lipschitz embedding}\right\}
\end{align}

\end{definition}

In \cite{Jo-Va11} (Theorem 1 and Proposition 3), Pierre-Nicolas Jolissaint and Alain Valette proved that for finite graphs the following inequality holds:

\begin{align}
c_p(X)\geq\frac{D(X)}{2}\left(\frac{|X|}{|E|\lambda_1^{(p)}}\right)^{\frac{1}{p}}
\end{align}
where $E$ denotes the edge set and
\begin{align}\label{eq:displacement}
D(X)=\max_{\alpha\in Sym(X)}\min_{x\in X}d(x,\alpha(x))
\end{align}

We want to extend this inequality to at least all finite metric spaces. Let $(X,d)$ be a locally finite metric space and let $X_1,\ldots X_n$ be the partition of $X$ in path-connected components. The basic idea is clearly to apply P.N.Jolissaint-Valette's inequality on each of them, but unfortunately this application is not straightforward, since a path-connected component might not look like a graph (think, for instance, of the space $[-n,n]^2\setminus\{(0,0)\}\subseteq\mathbb Z^2$ equipped with the metric induced by the standard embedding into $\mathbb R^2$). So we have to be a bit careful to apply P.N.Jolissaint-Valette's argument.

\begin{remark}\label{rem:normalform}
{\rm Since the $\ell_p$-distortion does not depend on rescaling the metric and since we are going to work on each path-connected component separately, we can suppose that each $X_i$ is in normal form\footnote{Let $(X,d)$ be a locally finite path-connected metric space. Given $x\in X$, let $R_x$ be the radius of the smallest closed ball with the center $x$ containing at least two points. It is straightforward to prove that $R_x$ does not depend on $x$ and it is called \emph{step} of the space. We say that a locally finite path-connected metric space is in normal form if the metric is normalized in such a way that the step is $1$.}.}
\end{remark}

Since we are going to work on a fixed path-connected component, let us simplify the notation assuming directly that $X$ is finite path-connected metric space in normal form. At the end of this section it will be easy to put together all path-connected components.

Let $x,y\in X$, $x\neq y$ and let $x_0x_1\ldots x_{n-1}x_n$ be a continuous path joining $x$ and $y$ of minimal length $n$. Denote by $s$ the floor of $d(x,y)$, i.e. $s$ is the greatest positive integer smaller than or equal $d(x,y)$. Notice that $s\geq1$, since $X$ is in normal form. Denote by $\mathcal P(x,y)$ the set of coverings $P=\{p_1,\ldots,p_s\}$ of the set $\{0,1,\ldots,n\}$ such that\footnote{Observe that each $p_i$ is a subset of $\{0,1,\ldots,n\}$.}
\begin{itemize}
\item If $a\in p_i$ and $b\in p_{i+1}$, then $a\leq b$,
\item the greatest element of $p_i$ is equal to the smallest element of $p_{i+1}$.
\end{itemize}
We denote by $p_i^-$ and $p_i^+$ respectively the smallest and the greatest element of $p_i$.

Now we introduce the following set

\begin{align}\label{eq:edgesmetric}
E(X)=\left\{(e^-,e^+)\in X\times X: \exists x,y\in X, p\in\mathcal P(x,y):e^-=p_i^-, e^+=p_i^+\right\}
\end{align}

\begin{remark}\label{rem:metricedgesvsgraphedges}
{\rm If $X=(V,E)$ is a finite connected graph equipped with the shortest path metric, then $E(X)=E$. Indeed in this case $s=n$ and so the only coverings belonging to $\mathcal P(x,y)$ have the shape $p_i=\{x_{i-1},x_i\}$, where the $x_i$'s are taken along a shortest path joining $x$ and $y$.}
\end{remark}

We define a metric analogue of the $p$-spectral gap: for $1\leq p<\infty$, we set

\begin{align}\label{eq:spectralgap}
\lambda_1^{(p)}=\inf\left\{\frac{\sum_{e\in E(X)}|f(e^+)-f(e^-)|^p}{\inf_{\alpha\in\mathbb R}\sum_{x\in X}|f(x)-\alpha|^p}\right\}
\end{align}

where the infimum is taken over all functions $f\in\ell^p(X)$ which are not constant.

\begin{lemma}\label{lem:jolivalette}
Let $(X,d)$ be a finite path-connected metric space in normal form.
\begin{enumerate}
\item For any permutation $\alpha\in Sym(X)$ and $F:X\rightarrow\ell^p(\mathbb N)$, one has
$$
\sum_{x\in X}||F(x)-F(\alpha(x))||_p^p\leq2^p\sum_{x\in X}||F(x)||_p^p
$$
\item For any bi-lipschitz embedding $F:X\rightarrow\ell^p(\mathbb N)$, there is another bi-lipschitz embedding $G:X\rightarrow\ell^p(\mathbb N)$ such that $||F||_{Lip}||F^{-1}||_{Lip}=||G||_{Lip}||G^{-1}||_{Lip}$ and
$$
\sum_{x\in X}||G(x)||_p^p\leq\frac{1}{\lambda_1^{(p)}}\sum_{e\in E(X)}||G(e^+)-G(e^-)||_p^p
$$
\end{enumerate}
\end{lemma}
\begin{proof}
\begin{enumerate}
\item This proof is absolutely the same as the proof of Lemma 1 in \cite{Jo-Va11}.
\item Observe that the construction of $G$ made in \cite{Gr-No10} is purely algebraic and so we can apply it. The inequality just follows from our definition of $\lambda_1^{(p)}$.
\end{enumerate}
\end{proof}

Now we have to prove a version for metric spaces of a useful lemma already proved by Linial and Magen for finite connected graph (see \cite{Li-Ma00}, Claim 3.2). We need to introduce a number that measures how far is the metric space to be a graph. We set

\begin{align}\label{eq:dis}
d(X)=\max_{e\in E(X)}d(e^-,e^+)
\end{align}

We have told that $d(X)$ measures how far $X$ is far from being a graph. Indeed, the following proposition holds:

\begin{proposition}\label{prop:graphvsconnectedcomponents}
The followings are equivalent:
\begin{enumerate}
\item $d(X)=1$,
\item The distance of any two points $x,y\in X$ is exactly the length of the shortest path connecting $x,y$
\end{enumerate}
\end{proposition}

\begin{proof}
In one sense the thesis is trivial: if $X$ is a finite graph, then $E(X)=E$ (by Remark \ref{rem:metricedgesvsgraphedges}) and $d(X)=1$, since the space is supposed to be in normal form. Conversely, suppose that $d(X)=1$, choose two distinct points $x,y\in X$ and let $x_0x_1\ldots x_{n-1}x_n$ be a continuous path of minimal length such that $x_0=x$ and $x_n=y$. Of, course $s\leq d(x,y)\leq n$. So, it suffices to prove that $s=n$. In order to do that, suppose that $s<n$ and observe that every $p\in\mathcal P(x,y)$ would contain some $p_i$ containing at least three points $p_{i}^-=x_{i-1},x_i,x_{i+1}=p_i^+$ (we suppose that they are exactly three, since the general case is similar). Since $X$ is in normal form, it follows that $d(x_{i-1},x_i)=d(x_i,x_{i+1})=1$. Now, suppose that $d(X)=1$, it follows that also $d(x_{i-1},x_{i+1})=1$ and then the path $x_0\ldots x_{i-1}x_{i+1}\ldots x_n$ is still a continuous path connecting $x$ with $y$, contradicting the minimality of the length of the previous path.
\end{proof}

We are now able to prove the generalization of Linial-Magen's lemma that we need.

\begin{lemma}\label{lem:linialmagen}
Let $(X,d)$ be a finite path-connected metric space in normal form and $f:X\rightarrow\mathbb R$. Then
\begin{align}
\max_{x\neq y}\frac{|f(x)-f(y)|}{d(x,y)}\leq\max_{e\in E(X)}|f(e^+)-f(e^-)|\leq d(X)\max_{x\neq y}\frac{|f(x)-f(y)|}{d(x,y)}
\end{align}
\end{lemma}

\begin{proof}
Let us prove only the first inequality, since the second will be trivial \emph{a posteriori}. Let $x,y\in X$ where the maximum in the left hand side is attained and let $x_0x_1\ldots x_{n-1}x_n$ be a continuous path of minimal length connecting $x$ with $y$. Let $p\in\mathcal P(x,y)$ and let $k$ be an integer such that
$$
|f(p_k^+)-f(p_k^-)|\geq|f(p_i^+)-f(p_i^-)|
$$
for all $i$. It follows
$$
|f(p_k^+)-f(p_k^-)|=\frac{s|f(p_k^+)-f(p_k^-)|}{s}\geq\frac{\sum_{i=1}^s|f(p_i^+)-f(p_i^-)|}{s}\geq
$$
Now we use the fact that the covering $p$ is made exactly by $s$ sets. It follows that $x$ and $y$ belong to the union of the $p_i$'s and we can use the triangle inequality and obtain
$$
\geq\frac{|f(x)-f(y)|}{s}\geq\frac{|f(x)-f(y)|}{d(x,y)}
$$
\end{proof}

We are now ready to prove the main result of this section.

\begin{theorem}\label{th:jolivale}
Let $(X,d)$ be a finite path-connected metric space in normal form. For all $1\leq p<\infty$, one has
\begin{align}
c_p(X)\geq\frac{D(X)}{2d(X)}\left(\frac{|X|}{|E(X)|\lambda_1^{(p)}}\right)^{\frac{1}{p}}
\end{align}
where
\begin{align}\label{eq:displacement}
D(X)=\max_{\alpha\in Sym(X)}\min_{x\in X}d(x,\alpha(x))
\end{align}

\end{theorem}

\begin{proof}
Let $G$ be a bi-lipschitz embedding which verifies the second condition in Lemma \ref{lem:jolivalette} and let $\alpha$ be a permutation of $X$ without fixed points. Let $\rho(\alpha)=\min_{x\in X}d(x,\alpha(x))$. One has
\begin{align*}
\frac{1}{||G^{-1}||_{Lip}^p}\\
&=\min_{x\neq y}\frac{||G(x)-G(y)||_p^p}{d(x,y)^p}\\
&\leq\min_{x\in X}\frac{||G(x)-G(\alpha(x))||_p^p}{d(x,\alpha(x))^p}\\
&\leq\frac{1}{\rho(\alpha)^p}\min_{x\in X}||G(x)-G(\alpha(x))||_p^p\\
&\leq\frac{1}{\rho(\alpha)^p|X|}\sum_{x\in X}||G(x)-G(\alpha(x))||_p^p
\end{align*}
Now apply the first statement of Lemma \ref{lem:jolivalette}:
\begin{align*}
\leq\frac{2^p}{\rho(\alpha)^p|X|}\sum_{x\in X}||G(x)||_p^p
\end{align*}
Now apply the second statement of Lemma \ref{lem:jolivalette}:
\begin{align*}
\leq\frac{2^p}{\rho(\alpha)^p|X|\lambda_1^{(p)}}\sum_{e\in E(X)}||G(e^+)-G(e^-)||_p^p
\leq \frac{2^p|E(X)|}{\rho(\alpha)^p|X|\lambda_1^{(p)}}\max_{e\in E(X)}||G(e^+)-G(e^-)||_p^p
\end{align*}
Now apply Lemma \ref{lem:linialmagen}:
\begin{align*}
\leq\frac{2^pd(X)^p|E(X)|||G||_{Lip}^p}{\rho(\alpha)^p|X|\lambda_1^{(p)}}
\end{align*}
Now recall the definitions in Equation \ref{eq:distortion} and \ref{eq:displacement} and just re-arrange the terms to get the desired inequality.
\end{proof}

Notice that $d(X)\geq1$ and so the inequality gets worse when the metric space is not a graph.

\begin{corollary}\label{cor:jolivale}
Let $X$ be a metric space such that every path-connected component $X_i$ is finite. One has
\begin{align}
c_p(X)\geq\sup_i\frac{D(X_i)}{2d(X_i)}\left(\frac{|X_i|}{|E(X_i)|\lambda_1^{(p)}}\right)^{\frac{1}{p}}
\end{align}
\end{corollary}

\begin{proof}
Just observe that if $X_i$ is a partition of $X$ then $c_p(X)\geq\sup_ic_p(X_i)$.
\end{proof}

\end{document}